\documentclass[11pt,a4paper]{amsart}
\usepackage{graphicx}
\usepackage{amsmath,amsfonts}
\usepackage{amsthm,amssymb,latexsym}
\usepackage[active]{srcltx}
\usepackage{graphicx,amsmath,amsfonts,amssymb,verbatim,enumerate,amsthm,url,multicol,mathtools,tikz,xspace} 
\usepackage[colorlinks]{hyperref}
\usepackage[english]{babel}
\usepackage[utf8]{inputenc}
\newtheorem{theorem}{Theorem}[section]

\newtheorem{proposition}[theorem]{Proposition}

\newtheorem{lemma}[theorem]{Lemma}
\newtheorem{corollary}[theorem]{Corollary}

\newtheorem{remark}[theorem]{Remark}

\newcommand{\A}{\mathbb A}
\newcommand{\F}{\mathbb F}
\newcommand{\K}{\mathbb K}

\newcommand{\Pp}{\mathbb P}
\newcommand{\Lt}{\operatorname{Lt}}

\newcommand{\fp}{\F_{\hskip-0.7mm p}}
\newcommand{\fq}{\F_{\hskip-0.7mm q}}

\newcommand{\cfq}{\overline{\F}_{\hskip-0.7mm q}}

\begin{document}
\title[Markoff-Hurwitz equations]{Estimates on the number of rational solutions of Markoff-Hurwitz equations over finite fields}

\author[M. Abdon]{
Miriam Abdon${}^{1}$}
\author[D. Oliveira]{
Daniela Alves de Oliveira${}^{2}$}
\author[J. Capaverde]{
Juliane Capaverde${}^{3}$}
\author[M. P\'erez]{
Mariana P\'erez${}^{4,5}$}
\author[M. Privitelli]{
Melina Privitelli${}^{4,5}$}

\address{${}^{1}$ Instituto de Matemática e Estatística, Universidade Federal Fluminense, Niterói, RJ 24210-201, Brazil}
\email{miriam\textunderscore abdon@id.uff.br}

\address{${}^{2}$ Instituto de Ciências Matemáticas e de Computação, Universidade de São Paulo (USP), São Carlos, Brazil.}
\email{danielaalvesoliveira@gmail.com}

\address{${}^{3}$ Departamento de Matemática Pura e Aplicada, Universidade Federal do Rio Grande do Sul, Avenida Bento Gonçalves, 9500, Porto Alegre, RS 91509-900, Brazil.}
\email{juliane.capaverde@ufrgs.br}

\address{${}^{4}$ Consejo Nacional de Investigaciones Científicas y Técnicas (CONICET),
Ar\-gentina}

%

\address{${}^{5}$Universidad Nacional de Hurlingham, Laboratorio de Investigación en Geometría y Álgebra (LIGA)\\ Av. Gdor. Vergara 2222 (B1688GEZ), Villa Tesei,
Buenos Aires, Argentina}
\email{\{mariana.perez,melina.privitelli\}@unahur.edu.ar}

\thanks{The authors were partially supported by the grants
PIUNAHUR 2023 80020230100
003HU, PIBAA 2022-2023-28720210100460CO, PIBAA 2022-2023-28720210100422CO, and FAPESP 2022/14004-7.}


\begin{abstract}
    Let $N$ denote the number of solutions to the generalized Markoff-Hurwitz-type equation
    \[(a_1X_1^m+\cdots + a_nX_n^m+a)^k=bX_1\cdots X_n \]
    over the finite field $\mathbb{F}_q$, where $m,k$ are positive integers, and $a,b,a_i\in \mathbb{F}_q^*$ for $i=1,\dots, n$, with $k,m\ge 2$ and $n\ge 3$. Using techniques from algebraic geometry, we provide an estimate for $N$ and establish conditions under which the equation admits solutions where all $X_i$ are nonzero.
\end{abstract}

\maketitle

\section{Introduction}
Let $\fq$ be the finite field of $q=p^s$ elements, and let \(\mathbb{F}_q^* = \mathbb{F}_q \setminus \{0\}\).  Let $N$ denote the number of $\fq$-rational solutions (i.e., solutions with coordinates in $\fq$) of a polynomial equation $f=0$, where $f\in \fq[X_1,\ldots,X_n]$. The problem of determining the value of $N$ is a fundamental question in the study of finite fields.  In general, deriving an explicit formula for $N$ is highly nontrivial. There is no known explicit formula for $N$ in general. In view of this, it would be interesting to establish estimates for $N$ and give conditions under which the equation has a nontrivial solution (i.e., a solution with $(x_1, \ldots, x_n) \neq (0, \ldots, 0)$).

Markoff-Hurwitz equations are a class of Diophantine equations given by  
\begin{equation} \label{exp2}
X_1^2 + X_2^2 + \cdots + X_n^2 = bX_1X_2\cdots X_n,  
\end{equation}  
where \( n \) and \( b \) are positive integers satisfying \( n \geq 3 \).   These equations were initially studied by A. Markoff \cite{Markoff} in the specific case \( n = b = 3 \), motivated by their connection to Diophantine approximation. Subsequently, A. Hurwitz \cite{Hurwitz} extended the study to more general forms of these equations.

In \cite{Carlitz}, L. Carlitz initiated the study of Markoff-Hurwitz-type equations over finite fields. He expressed the number $N$ in terms of Jacobsthal sums for $n=3,4$ and $p>2$. 

The generalized Markoff-Hurwitz-type equations are of the form
\begin{equation*}\label{Baoulina}
(a_1X_1^{m_1}+\cdots +a_nX_n^{m_n}+a)^k=bX_1^{k_1}\ldots X_n^{k_n},
\end{equation*}
 where $n,m_1,\ldots,m_n, k_1, \ldots, k_n$, $k$ are positive integers, $a, b \in \fq$ and $a_i \in \fq^*$, for $1\leq i \leq n$. These equations have been extensively studied in the special case $a=0$ and $k=1$. Moreover, several papers provide explicit formulas of $N$ for this type of equations in very particular cases. In \cite{Carlitz} explicit formulas of $N$ were provided for the case $k=2$, $k_i=m_i=a_i=1$, $1\leq i\leq n $, $n=3,4$  and $p>2$. 
In a series of papers \cite{Baoulina1, Baoulina2, Baoulina3}, I. Baoulina studied  equations of the form
\begin{equation}\label{eqgeneralized}
    a_1X_1^{m_1} + a_2X_2^{m_2} + \cdots + a_nX_n^{m_n} = bX_1\cdots X_n,
\end{equation}
where \(a_i, b \in \mathbb{F}_q^*\), \(m_i\) are positive integers such that \(m_i\) divides \(q-1\) for \(i = 1, \dots, n\), and \(n \geq 2\). Her work employed techniques from character sums and combinatorics.

In \cite{Cao}, W. Cao considered the generalized equation
\begin{equation}\label{cao1}
    a_1X_1^{m_1} + a_2X_2^{m_2} + \cdots + a_nX_n^{m_n} = bX_1^{k_1} \cdots X_n^{k_n},
\end{equation}  
where \(a_i, b \in \mathbb{F}_q^*\) and \(m_i, k_i\) are positive integers for \(i = 1, \dots, n\), with \(n \geq 2\). By employing a tool augmented with degree matrices, it was shown that all generalized Markoff-Hurwitz-type equations in \(n\) variables, with associated invertible augmented degree matrices satisfying certain conditions, have the same number of solutions. Moreover, the expression for the number of solutions of \eqref{cao1} becomes significantly simplified under these conditions.
In \cite{Pan2014} the authors considered the equation \eqref{Baoulina} with $k\geq 2$ and they obtained an explicit estimate of $N$   which holds under certain relation between the exponents and $q$.
In \cite{Melina&Mariana2020}, the authors concentrated in the case $a\neq 0$ and $k=1$. More precisely, they obtain an estimate on $N^*$, the number of $\fq$--rational solutions with the condition that $x_1\cdots x_n \neq 0$. This estimate improves Mordell's  \cite{Mo63} for the case $m:=m_1=\cdots = m_n$ since the result holds without conditions on the characteristic of the field and it requires that $m>k_1+\cdots+k_n$ instead of $k_1= \cdots= k_n=1$. 
Also, the estimate improves Mordell's by determining an extra term in the asymptotic development of $N^*$ in terms of $q$.

In this paper, we study a further generalization of Markoff-Hurwitz-type equations:  
\begin{equation}\label{MH}
      (a_1X_1^{m} + \dots + a_nX_n^m + a)^k = bX_1 \cdots X_n,
\end{equation}  
where \(a, b, a_i \in \mathbb{F}_q^*\), for \(i = 1, \dots, n\),  and \(m, n, k\) are positive integers with \(m, k \geq 2\) and $n\geq 3$, such that \(p\) does not divide \(mk\) and \(mk > n\).  By employing techniques from algebraic geometry, we provide an estimate for \(N\) and determine whether Eq.~\eqref{MH} admits solutions where all coordinates are nonzero, assuming that $m$ and $q-1$ are relatively prime.

  The paper is organized as follows. In Section $2$, we review the algebraic geometry notions used throughout the paper. In Section $3$, we study the geometric properties of the hypersurface associated with Eq. \eqref{MH}. In Section $4$, we obtain estimates for the number of $\fq$--rational solutions and establish existence results for this equation.

\section{Basic notions of algebraic geometry}
\label{sec: notions of algebraic geometry}
In this section we collect the basic definitions and facts of
algebraic geometry that we need in the sequel. We use standard
notions and notations which can be found in, e.g., \cite{Kunz85},
\cite{Shafarevich94}.

Let $\K$ be any of the fields $\fq$ or $\cfq$. We denote by $\A^r$
the affine $r$--dimensional space $\cfq{\!}^{r}$ and by $\Pp^r$ the
projective $r$--dimensional space over $\cfq{\!}^{r+1}$. Both spaces
are endowed with their respective Zariski topologies over $\K$, for
which a closed set is the zero locus of a set of polynomials in
$\K[X_1,\ldots, X_r]$, or of a set of homogeneous polynomials in
$\K[X_0,\ldots, X_r]$.

A subset $V\subset \Pp^r$ is a {\em projective variety defined over}
$\K$ (or a projective $\K$--variety for short) if it is the set of
common zeros in $\Pp^r$ of homogeneous polynomials $F_1,\ldots, F_m
\in\K[X_0 ,\ldots, X_r]$. Correspondingly, an {\em affine variety of
	$\A^r$ defined over} $\K$ (or an affine $\K$--variety) is the set of
common zeros in $\A^r$ of polynomials $F_1,\ldots, F_{m} \in
\K[X_1,\ldots, X_r]$. We think a projective or affine $\K$--variety
to be equipped with the induced Zariski topology. We shall denote by
$\{F_1=0,\ldots, F_m=0\}$ or $V(F_1,\ldots,F_m)$ the affine or
projective $\K$--variety consisting of the common zeros of
$F_1,\ldots, F_m$. 

In the remaining part of this section, unless otherwise stated, all
results referring to varieties in general should be understood as
valid for both projective and affine varieties.

A $\K$--variety $V$ is {\em irreducible} if it cannot be expressed
as a finite union of proper $\K$--subvarieties of $V$. Further, $V$
is {\em absolutely irreducible} if it is $\cfq$--irreducible as an
$\cfq$--variety. Any $\K$--variety $V$ can be expressed as an
irredundant union $V=\mathcal{C}_1\cup \cdots\cup\mathcal{C}_s$ of
irreducible (absolutely irreducible) $\K$--varieties, unique up to
reordering, called the {\em irreducible} ({\em absolutely
	irreducible}) $\K$--{\em components} of $V$.

For a $\K$--variety $V$ contained in $\Pp^r$ or $\A^r$, its {\em
	defining ideal} $I(V)$ is the set of polynomials of $\K[X_0,\ldots,
X_r]$, or of $\K[X_1,\ldots, X_r]$, vanishing on $V$. The {\em
	coordinate ring} $\K[V]$ of $V$ is the quotient ring
$\K[X_0,\ldots,X_r]/I(V)$ or $\K[X_1,\ldots,X_r]/I(V)$. The {\em
	dimension} $\dim V$ of $V$ is the length $n$ of a longest chain
$V_0\varsubsetneq V_1 \varsubsetneq\cdots \varsubsetneq V_n$ of
nonempty irreducible $\K$--varieties contained in $V$.
We say that $V$ has {\em pure dimension} $n$ if every irreducible
$\K$--component of $V$ has dimension $n$. A $\K$--variety of $\Pp^r$
or $\A^r$ of pure dimension $r-1$ is called a $\K$--{\em	hypersurface}. A $\K$--hypersurface of $\Pp^r$ (or $\A^r$) can also
be described as the set of zeros of a single nonzero polynomial of
$\K[X_0,\ldots, X_r]$ (or of $\K[X_1,\ldots, X_r]$).

The {\em degree} $\deg V$ of an irreducible $\K$--variety $V$ is the
maximum of $|V\cap L|$, considering all the linear spaces $L$ of
codimension $\dim V$ such that $|V\cap L|<\infty$. More generally,
following \cite{Heintz83} (see also \cite{Fulton84}), if
$V=\mathcal{C}_1\cup\cdots\cup \mathcal{C}_s$ is the decomposition
of $V$ into irreducible $\K$--components, we define the degree of
$V$ as
$$\deg V:=\sum_{i=1}^s\deg \mathcal{C}_i.$$
The degree of a $\K$--hypersurface $V$ is the degree of a polynomial
of minimal degree defining $V$. 
%

We shall use the following {\em B\'ezout inequality} (see
\cite{Fulton84, Heintz83, Vogel84}): if $V$ and $W$
are $\K$--varieties of the same ambient space, then
\begin{equation}\label{eq: Bezout}
\deg (V\cap W)\le \deg V \cdot \deg W.
\end{equation}

Let $V\subset\A^r$ be a $\K$--variety, $I(V)\subset \K[X_1,\ldots,
X_r]$ its defining ideal and $x$ a point of $V$. The {\em dimension}
$\dim_xV$ {\em of} $V$ {\em at} $x$ is the maximum of the dimensions
of the irreducible $\K$--components of $V$ containing $x$. If
$I(V)=(F_1,\ldots, F_m)$, the {\em tangent space} $\mathcal{T}_xV$
to $V$ at $x$ is the kernel of the Jacobian matrix $(\partial
F_i/\partial X_j)_{1\le i\le m,1\le j\le r}(x)$ of $F_1,\ldots, F_m$
with respect to $X_1,\ldots, X_r$ at $x$. We have
$\dim\mathcal{T}_xV\ge \dim_xV$ (see, e.g., \cite[page
94]{Shafarevich94}). The point $x$ is {\em regular} if
$\dim\mathcal{T}_xV=\dim_xV$; otherwise, $x$ is called {\em
	singular}. The set of singular points of $V$ is the {\em singular
	locus} of $V$; it is a closed $\K$--subvariety of
$V$. A variety is called {\em nonsingular} if its singular locus is
empty. For projective varieties, the concepts of tangent space,
regular and singular point can be defined by considering an affine
neighborhood of the point under consideration.

%
%
\subsection{Rational points}
Let $\Pp^r(\fq)$ be the $r$--dimensional projective space over $\fq$
and $\A^r(\fq)$ the $r$--dimensional $\fq$--vector space $\fq^n$.
For a projective variety $V\subset\Pp^r$ or an affine variety
$V\subset\A^r$, we denote by $V(\fq)$ the set of $\fq$--rational
points of $V$, namely $V(\fq):=V\cap \Pp^r(\fq)$ in the projective
case and $V(\fq):=V\cap \A^r(\fq)$ in the affine case. For an affine
variety $V$ of dimension $n$ and degree $\delta$, we have the
following bound (see, e.g., \cite[Lemma 2.1]{CaMa06}):
\begin{equation*}\label{eq: upper bound -- affine gral}
|V(\fq)|\leq \delta\, q^n.
\end{equation*}
On the other hand, if $V$ is a projective variety of dimension $n$
and degree $\delta$, then we have the following bound (see
\cite[Proposition 12.1]{GhLa02a} or \cite[Proposition 3.1]{CaMa07};
see \cite{LaRo15} for more precise upper bounds):
\begin{equation*}\label{eq: upper bound -- projective gral}
|V(\fq)|\leq \delta\, p_n,
\end{equation*}
where $p_n:=q^n+q^{n-1}+\cdots+q+1=|\Pp^n(\fq)|$.
%
%
\subsection{Complete intersections}\label{subsec: complete intersections}
A \emph{set theoretic complete intersection} is an affine $\K$-variety $ V (F_1,\ldots,F_m)\subset\A^r$ or a projective $\K$-variety $V (F_1,\ldots,F_m)\subset\Pp^r$ defined by $m \leq r$ polynomials $F_1, \ldots,F_m \in \K[X_1, \ldots, X_r]$ or homogeneous polynomials $F_1, \ldots,F_m$ in $\mathbb{K}[X_0,\ldots,X_r]$, which is of pure dimension $r-m$. If in addition $(F_1, \ldots,F_m)$ is a radical ideal of $\K[X_1, \ldots, X_r]$ or $\K[X_0, \ldots,X_r]$, then we say that $V (F_1,\ldots,F_m)$ is an \emph{ideal–theoretic complete intersection} (complete intersection for short). 

Elements $F_1,\ldots, F_m$ in $\mathbb{K}[X_1,\ldots,X_r]$ or
$\mathbb{K}[X_0,\ldots,X_r]$ form a \emph{regular sequence} if the ideal $(F_1, \ldots, F_m)$  they define in $\K[X_1, \ldots,X_r]$ or $\mathbb{K}[X_0, \ldots, X_r]$ is a proper ideal, $F_1$
is nonzero and, for $2 \leq i \leq m$,  $F_i$ is  neither zero nor a zero divisor in  $\mathbb{K}[X_1,\ldots,X_r]/ (F_1,\ldots,F_{i-1})$ or
$\mathbb{K}[X_0,\ldots,X_r]/ (F_1,\ldots,F_{i-1})$.  In such a case, the (affine or projective) variety $ V (F_1,\ldots,F_m)$ they define is a set theoretic complete intersection.

In what follows we
shall use the following result (see, e.g., \cite[Chapitre 3,
Remarque 2.2]{Lejeune84}).
\begin{lemma}\label{lemma: reg seq and complete int}
If $F_1,\ldots, F_m$ are  $m\le r+1$ homogeneous polynomials in $\K[X_0
,\ldots, X_r]\setminus \K$, the following conditions are equivalent:
   \begin{enumerate}[{\normalfont (a)}]
     \item $F_1 ,\ldots, F_m$ define a set theoretic complete intersection of $\mathbb{P}^r$.
     \item $F_1 ,\ldots, F_m$  is a regular sequence in $\K[X_0 ,\ldots, X_r]$.
   \end{enumerate}
\end{lemma}

\begin{remark}
If $F_1,\ldots, F_m \in \K[X_1
,\ldots, X_r]$, $m\leq r$, only {\normalfont (b)} implies {\normalfont (a)} holds.
\end{remark}
 
%
%

If $V\subset\Pp^r$ is a complete intersection defined over $\K$ of
dimension $r-m$, and $F_1 ,\ldots, F_m$ is a system of homogeneous
generators of $I(V)$, the degrees $d_1,\ldots, d_m$ depend only on
$V$ and not on the system of generators. Arranging the $d_i$ in such
a way that $d_1\geq d_2 \geq \cdots \geq d_m$, we call $(d_1,\ldots,
d_m)$ the {\em multidegree} of
$V$. In this case, a stronger version of 
\eqref{eq: Bezout} holds, called the {\em B\'ezout theorem} (see,
e.g., \cite[Theorem 18.3]{Harris92})
\begin{equation*}\label{eq: Bezout theorem}
\deg V=d_1\cdots d_m.
\end{equation*}
A complete intersection $V$ is called {\em normal} if it is {\em
	regular in codimension 1}, that is, the singular locus
$\mathrm{Sing}(V)$ of $V$ has codimension at least $2$ in $V$,
namely $\dim V-\dim \mathrm{Sing}(V)\ge 2$ (actually, normality is a
general notion that agrees on complete intersections with the one we
define here).

%


In what follows, we present Theorem 15.15 from the book \cite{Eisenbud95} , which is central to proving that the varieties we shall consider are complete intersections. Fixed a monomial order on \( K[X_1, \dots, X_n] \), we denote by \( \Lt(g) \) the leading term of a polynomial \( g \in K[X_1, \dots, X_n] \) with respect to the fixed monomial order.


\begin{proposition} \cite[Proposition 15.15]{Eisenbud95}\label{Prop1}
Let $S=K[X_1,\ldots,X_n]$ the polynomial ring with coefficients in a field $k$. We fix a monomial order in $S$ and We consider $F$, a free $S$-module with basis and monomial order compatible 
with a given monomial order on $S$. If $M \subset F$ is any submodule 
and $h_1, \ldots, h_u \in S$ are such that 
$\Lt(h_1), \ldots, \Lt(h_u)$ is a regular sequence 
on $F / \Lt(M)$, then $h_1, \ldots, h_u$ is a regular sequence 
on $F/M$.\end{proposition}

Taking $F=S=\fq[X_1,\ldots,X_n]$ and $M=\{\boldsymbol{0}\}$ in the above proposition, we obtain the following result.
\begin{proposition} Let $\fq[X_1,\cdots,X_n]$ endowed with a fixed monomial order. 
   Suppose that $h_1, \dots, h_u \in \mathbb{F}_q[X_1,\dots,X_n]$ are such that their leading terms 
$\Lt(h_1), \dots, \Lt(h_u)$ form a regular sequence in $\mathbb{F}_q[X_1,\dots,X_n]$.  
Then $h_1, \dots, h_u$ themselves form a regular sequence in $\mathbb{F}_q[X_1,\dots,X_n]$.
\end{proposition}

\section{Generalized Markoff-Hurwitz-type equations}

Let $n \geq 3, m\ge 2$ and $k\ge2 $ be integers. Let $a,b, a_i\in \fq\setminus\{0\},$ for all $1\le i \le n$. Suppose that char($\fq$) does not divide $mk$ and that $mk>n$. We consider the equation defined in \eqref{MH}, namely
\begin{equation*}
   (a_1X_1^{m}+ \dots +a_nX_n^m+a)^k=bX_1\dots X_n.
\end{equation*}

Let $N$ denote the number of $\fq$-rational solutions of $\eqref{MH}$. Throughout this paper, we set $$f = (a_1X_1^m+\cdots+a_nX_n^m+a)^k-bX_1\cdots X_n$$ and let $V_f=V(f)  \subset \A^n$ be the affine hypersurface defined by $f$. We denote by $\mathrm{sing}(V_f)$ the set of singular points of $V_f$, that is, the set of singular points of the hypersurface. 

    \begin{theorem} \label{prop: singular locus V afin}
The dimension of $\mathrm{sing}(V_f)$ is at most $n-3$.
\end{theorem}

\begin{proof}
    The partial derivatives of $f$ are given by
    \[\frac{\partial f}{\partial X_j} = k(a_1X_1^{m}+\cdots +a_nX_n^m+a)^{k-1} \cdot m a_jX_j^{m-1} - bX_1\cdots X_{j-1}\cdot  X_{j+1}\cdots  X_n\]
    for $1\le j \le n$.  Let $\boldsymbol{x}\in \mathrm{sing}(V_f)$. Then $\boldsymbol{x}\in V(f,\frac{\partial f}{\partial X_1}, \dots, \frac{\partial f}{\partial X_n})$. Therefore $\boldsymbol{x}$ satisfies the following equations:

  $$(a_1X_1^m+\cdots+a_nX_n^m+a)^k-bX_1\cdots X_n=0,$$
       $$E_j:\,=\, kma_jX_j^{m-1}\,(a_1X_1^{m}+\cdots +a_nX_n^m+a)^{k-1}- bX_1\cdots X_{j-1} X_{j+1}\cdots X_n=0,$$
for $1\leq j \leq n$.
    
    Suppose that there exists $1\leq j \leq n$ such that $x_j=0$. Then $\boldsymbol{x}\in V_j:=V(X_j, X_{\ell}, a+\sum_{i \neq j,\ell} a_iX_i^m)$ for some $\ell \neq j$.  It is easy to see that $\{X_j,X_l,a+\sum_{i \neq j,\ell} a_iX_i^m\}$ form a regular sequence in $\mathbb{F}_q[X_1,\ldots,X_n]$ which implies that  $\dim (V_j)=n-3$ for all $1\leq j \leq n$.

    Now suppose that $x_j \neq 0$ for all $1 \leq j \leq n$. We define the polynomial $g=a_1X_1^{m}+\cdots +a_nX_n^m+a$ and we claim that $g(\boldsymbol{x}) \neq 0$. Indeed, if $g(\boldsymbol{x})=0$, since $f(\boldsymbol{x})=0$ we obtain that $bx_1\cdots x_n=0$, which is a contradiction.
    Therefore $\boldsymbol{x}$ satisfies the equations
   
  \begin{equation*}
  f(\boldsymbol{x}) - x_j\frac{\partial f}{\partial X_j}(\boldsymbol{x}) = g(\boldsymbol{x})^{k-1} \cdot(g(\boldsymbol{x})- km \cdot a_j x_j^{m})=0, \text{ for } 1 \leq j \leq n.
  \end{equation*}
    Since $g(\boldsymbol{x}) \neq 0$ deduce that     \[g(\boldsymbol{x})- km \cdot a_j x_j^{m} = 0, \text{ for } 1\leq j \leq n.\]

    On the other hand, observe that  $x_1E_1-x_jE_j=km(a_1x_1^m-a_jx_j^m)\cdot g(\boldsymbol{x})^{k-1}=0$, $2\leq j \leq n $. Taking into account that $g(\boldsymbol{x}) \neq 0$, we obtain that 
    \[a_1x_1^m - a_jx_j^m=0,\ \text{ for }2\leq j \leq n.\]
    Hence 
    \[g(\boldsymbol{x})- km \cdot a_1 x_1^{m} + \sum_{j=2}^n (a_1x_1^m - a_jx_j^m) = a_1(n-km) x_1^m +a = 0,\]
    and $\boldsymbol{x}\in W:=V(a_1(n-km) X_1^m +a,a_1X_1^m-a_2X_2^m,a_1X_1^m-a_3X_3^m)$, which has dimension $n-3$. Indeed, we consider in  $\mathbb{F}_q[X_1,\dots,X_n]$ the graded lexicographic order with $X_n>\cdots > X_1$. Then we have that
    
    \begin{itemize}
       \item  $\Lt(a_1(n-km) X_1^m +a)= a_1(n-km)X_1^m$, 
        \item $\Lt(a_1X_1^m-a_2X_2^m )= a_2X_2^m$, 
        \item $\Lt(a_1X_1^m-a_3X_3^m) = a_3X_3^m$.
    \end{itemize}
    
Since these leading terms are relatively prime in $\mathbb{F}_q[X_1,\dots,X_n]$, they form a Gröbner basis of the ideal that they generate.
    From Proposition \ref{Prop1} we conclude that $a_1(n-km) X_1^m +a, a_1X_1^m-a_2X_2^m$, $a_1X_1^m-a_3X_3^m$ form a regular sequence in $\mathbb{F}_q[X_1,\dots,X_n]$. Therefore  
    $W$ has dimension $n-3$, completing the proof.
\end{proof}

We also require information about the behavior of $V_f$ at ``infinity." For this purpose, we consider the projective closure $\mathrm{pcl}(V_f) \subset \mathbb{P}^n$. It is well known that $\mathrm{pcl}(V_f)$ is the $\mathbb{F}_q$-hypersurface of $\mathbb{P}^n$ defined by the homogenization $f^h\in\mathbb{F}_q[X_0,X_1,\dots,X_n]$ of the polynomial $f$ (see, e.g., \cite[\S I.5, Exercise 6]{Kunz85}). Indeed, $\mathrm{pcl}(V_f)=V(f^h)$ and $f^h =(a_1X_1^m+\cdots+a_n X_n^{m}+aX_0^m)^k-bX_0^{mk-n}X_1\cdots X_n.$

On the other hand, we shall need to study the geometric properties of the projective variety $V_f^{\infty}:=\mathrm{pcl}(V_f)\cap \{X_0=0\}\subset \Pp^{n-1}$.
\begin{proposition}\label{prop: singular locus V infinito}
 $V_f^{\infty}$ is a non-singular hypersurface of $\Pp^{n-1}$. 
\end{proposition}
 \begin{proof}
     We have that $V_f^{\infty}=V(g-a) \subset \Pp^{n-1}$.
    Let $\mathrm{sing}(V_f^{\infty})$ be the singular locus of $V_f^{\infty}$. For a given $\boldsymbol{x}\in \mathrm{sing}(V_f^{\infty})$, we have that
     $\boldsymbol{x}$ satisfies the equations
     $$a_1X_1^m+\cdots+a_n X_n^{m}=0,\quad \frac{\partial (g-a)}{\partial X_j}=a_j m X_j^{m-1}=0, \,\,\,\, 1\leq j \leq n.$$
     
     Hence, we deduce that $V_f^{\infty}$ has no singular points.
 \end{proof}
\begin{theorem}\label{P2}
 $\mathrm{pcl}(V_f)$ has singular locus of dimension at most $n-3$.
\end{theorem}
\begin{proof}
Let $\Sigma^{\infty}=\mathrm{sing}(\mathrm{pcl}(V_f))\cap \{X_0=0\} \subset \Pp^n$ be the set of singular points of $\mathrm{pcl}(V_f)$ lying in the hyperplane $\{X_0=0\}.$ By  \cite[Lemma 1.1]{GhLa02a} we have that $\Sigma^{\infty}\subset \mathrm{sing}(V_f^{\infty})$ then, taking into account Proposition \ref{prop: singular locus V infinito}, we obtain that $\mathrm{pcl}(V_f)$ does not have singular points at infinity. So, from Proposition \ref{prop: singular locus V afin}, we deduce that the singular locus of $\mathrm{pcl}(V_f)$ has dimension at most $n-3$. 
\end{proof}

\begin{corollary} \label{coro: absoluta irreducibilidad}
    The hypersurfaces $V_f$  and $V_f^{\infty}$ are absolutely irreducible.
\end{corollary}
    \begin{proof}
We prove that $V_f$ is absolutely irreducible. The proof that $V_f^{\infty}$ is absolutely irreducible 
is similar. Let us observe that $V_f$ is absolutely irreducible if and only if $\mathrm{pcl}(V_f)$ is absolutely irreducible. Suppose that $\mathrm{pcl}(V_f)$ is not absolutely irreducible. Then, it has a nontrivial decomposition into absolutely irreducible components:
\[\mathrm{pcl}(V_f) = \mathcal C_1\cup \dots \cup \mathcal C_s,\]
where $\mathcal C_1, \dots, \mathcal C_s$ are projective hypersurfaces in $\mathbb{P}^n$. Since $\mathcal C_i \cap \mathcal C_j \neq \emptyset$ and $\mathcal C_i, \mathcal C_j$ are absolutely irreducible, we have $\dim(\mathcal C_i \cap \mathcal C_j) = n - 2$.

From Theorem \ref{P2}, we have $\dim \mathrm{sing}(\mathrm{pcl}(V_f)) \le n-3$. On the other hand, we have $\mathcal C_i \cap \mathcal C_j \subset \mathrm{sing}(\mathrm{pcl}(V_f))$ for any $i \neq j$, which implies $\dim \mathrm{sing}(\mathrm{pcl}(V_f)) \ge n - 2$, which is a contradiction.
    \end{proof}

\section{Estimates on the number of $\fq$--rational solutions of Markoff-Hurwitz type equations}
We are now interested in estimating the number of $\fq$--rational points of $V_f$, where $f = (a_1X_1^m+\dots+a_nX_n^m+a)^k-bX_1\dots X_n$.  So we need to estimate the number of $\fq$--rational points of $\mathrm{pcl}(V_f)$ and $V_f^{\infty}$. To do this, we shall use an estimate  due to \cite{CaMaPr2015} on the number of $\mathbb{F}_q$--rational points of a normal projective hypersurface. In \cite[Corollary 8.3]{CaMaPr2015}, the authors proved that, for an absolutely irreducible $\fq$--hypersurface $V\subset\mathbb{P}^{n}$ of degree $d\geq 2$ and singular locus of codimension at least $2$ in $V$, the number $|V(\fq)|$ of $\mathbb{F}_q$--rational points of $V$ satisfies the estimate:
\begin{equation}\label{eq: estimate CMP normal var}
\big||V(\fq)|-p_{n-1}\big|\leq(d-1)(d-2)q^{n-3/2}
+14(d-1)^2d^2q^{n-2},
\end{equation}
where $p_{n-1}= q^{n-1}+\dots+q+1$.

 From Theorem \ref{P2} and Corollary \ref{coro: absoluta irreducibilidad}, we know that $\mathrm{pcl}(V_f)$ is an absolutely irreducible hypersurface of degree $mk$ and singular locus of dimension at most $n-3$.
Hence, from \eqref{eq: estimate CMP normal var} we have that:
\begin{equation}\label{eq: estimacion pcl(v) case 2}
\big||\mathrm{pcl}(V_f)(\fq)|-p_{n-1}\big|\leq (mk-1)(mk-2)q^{n-3/2}+14(mk-1)^2(mk)^2 q^{n-2}.
\end{equation}
On the other hand, from Proposition \ref{prop: singular locus V infinito} we have that $V_f^{\infty}$ is a nonsingular variety. We can apply the following result due to P. Deligne (see, e.g., \cite{De74}): for a nonsingular ideal--theoretic complete intersection $V \subset \Pp^n$ defined over $\fq$, of dimension $r$ and multidegree ${\bf{d}}=(d_1, \ldots, d_n)$, the following estimate holds:
\begin{equation}\label{estimacion Deligne}
\big||V(\fq)|-p_r| \leq b'_r(n,{\bf{d}})q^{r/2},
\end{equation}	
	where $ b'_r(n,{\bf{d}})$ is the primitive $r$--th Betti number of any nonsingular complete intersection of $\Pp^n$ of dimension $r$ and multidegree ${\bf{d}}$ (see, e.g., \cite[Theorem 4.1]{GhLa02a} for an explicit expression of $ b'_r(n,{\bf{d}})$ in terms of $n,r$ and ${\bf{d}}$).
 
 From Theorem $4.1$ and Example $4.3$ in \cite{GhLa02a}, we have that the primitive $r$--th Betti number of nonsingular hypersurfaces in $\Pp^n$ of dimension $r$ and degree $d$ satisfies: 
\begin{equation} \label{eq:upper bound sums betti numbers}
b_{r}'(n,d) \leq \frac{d-1}{d} \left( (d-1)^{n} - (-1)^{n} \right) \leq (d-1)^{n}. 
\end{equation} 
Thus
	\begin{equation}\label{estimation V infty}
	\big||V_f^{\infty}(\fq)|-p_{n-2}\big| \leq (mk-1)^{n-1} q^{(n-2)/2}.
	\end{equation}

Observe that $|V_f(\fq)|=|\mathrm{pcl}(V_f)(\fq)|-|V_f^{\infty}(\fq)|$. From \eqref{eq: estimacion pcl(v) case 2} and \eqref{estimation V infty}, we have that: 
\begin{align*}
\big||V_f(\fq)|-q^{n-1}\big|\leq & \big||\mathrm{pcl}(V_f)(\fq)|-p_{n-1}\big|+\big||V_f^{\infty}(\fq)|-p_{n-2}\big|\\
\leq &(mk-1)(mk-2)q^{n-3/2}+14(mk-1)^2(mk)^2 q^{n-2}\\
& +(mk-1)^{n-1} q^{(n-2)/2}\\
\leq & (mk)^{\max\{n-1,4\}}q^{\frac{n-2}{2}}\Big(q^{\frac{n-1}{2}}+14 q^{\frac{n-2}{2}}+1\Big)\\
\leq & 9
(mk)^{\max\{n-1,4\}}q^{n-3/2}.
\end{align*}

This leads to the main result of this section.

\begin{theorem}\label{estimate}
    Suppose $k,m\geq 2$, $n \geq 3$, $a,b, a_i\in \fq^*$, $\operatorname{char}(\fq)\nmid mk$, and $n<mk$. Then the number $N$ of $\fq$--rational solutions of equation \eqref{MH} satisfies the following estimate:
  $$\big|N-q^{n-1}\big|\leq  9(mk)^{\max\{n-1,4\}}q^{n-3/2}.$$
\end{theorem}

\subsection{The case with nonzero coordinates}
Let $N^*$ denote the number of $\mathbb{F}_q$--rational solutions with nonzero coordinates for equation \eqref{MH}. In what follows we obtain an estimate for $N^*$.

Let \( N^= \) denote the number of $\mathbb{F}_q$--rational solutions of \eqref{MH} with at least one coordinate equal to zero. Applying the inclusion-exclusion principle, we obtain
\[
N^= = \sum_{i=1}^n (-1)^{i+1} \binom{n}{i} N_i,
\]
where \( N_i \) represents the number of $\mathbb{F}_q$--rational solutions with $x_{j_1} = x_{j_2} = \cdots = x_{j_i} = 0$, for $1\leq j_1 < j_2 < \cdots < j_i \leq n$. Therefore,
\[
N^* = N - N^=.
\]
The following result will be useful for estimating \( N^* \).

\begin{proposition}\label{estimatediagonal}
    For \( 1 \le i \le n-1 \), we have \[
    \left| N_i - q^{n-i-1} \right| \le (m-1)^{n-i} q^{(n-i-2)/2} (1 + q^{1/2}).
    \]
\end{proposition}
\begin{proof}
    Let \( i \in \{1, \dots, n-1\} \) and suppose that \( x_1 = \dots = x_i = 0 \). Then, equation \eqref{MH} reduces to a diagonal equation in \( n-i \) coordinates with \( a \neq 0 \), specifically,
    \[
    a_{i+1}X_{i+1}^m + \dots + a_nX_n^m = -a.
    \]
    By Remark 4.4 in \cite{Melina&Mariana2020}, we obtain
    \begin{align*}
   \left|  N_i - q^{n-i-1} \right|& \le (m-1)^{n-i} q^{(n-i-2)/2} (1 + q^{1/2})\\
    & \leq 2m^{n-i}q^{(n-i-1)/2}. \qedhere
    \end{align*} 
    
\end{proof}

\begin{remark}\label{Remark}

Note that, when \( a \neq 0 \) and \( \gcd(m, q-1) = 1 \), the equation \( X^m = a \) has a unique solution in \( \mathbb{F}_q \). This implies \( N_{n-1} = 1 \). Additionally, observe that \( N_n = 0 \), as in this case all coordinates are zero and $a\neq0$.
\end{remark}

From this we can rewrite the expression for \( N^= \) as 
\begin{align*}
    N^= &=  \sum_{i=1}^{n} (-1)^{i+1}\binom{n}{i} N_i \\
    &= \sum_{i=1}^{n-2} (-1)^{i+1} \binom{n}{i} N_i + (-1)^n \binom{n}{n-1} N_{n-1} + (-1)^{n+1} \binom{n}{n} N_n \\
    &= \sum_{i=1}^{n-2} (-1)^{i+1} \binom{n}{i} N_i + (-1)^n \cdot n.
\end{align*}

From this relation, we obtain
\begin{align}\label{O1}\nonumber
    N^* &= N - N^= = N - \sum_{i=1}^{n-2} (-1)^{i+1} \binom{n}{i} N_i - (-1)^n \cdot n \\ \nonumber
    &= N + \sum_{i=1}^{n-2} (-1)^i \binom{n}{i} N_i - (-1)^n \cdot n \\
    &= (N - q^{n-1}) + \sum_{i=1}^{n-2} (-1)^{i+1} \binom{n}{i} (N_i - q^{n-i-1}) \\ \notag
    &\quad + \sum_{i=0}^{n-2} (-1)^i \binom{n}{i} q^{n-i-1}  - (-1)^n \cdot n.
\end{align}

First, we can write
\begin{align*}
    \sum_{i=0}^n (-1)^i \binom{n}{i} q^{n-i-1} &= \frac{1}{q} \sum_{i=0}^n (-1)^i \binom{n}{i} q^{n-i} = \frac{(q-1)^n}{q}.
\end{align*}
Using this expression, we obtain the following
\begin{align}\label{P1}\nonumber 
    \sum_{i=0}^{n-2} (-1)^i \binom{n}{i} q^{n-i-1} &= \frac{1}{q} \sum_{i=0}^{n-2} (-1)^i \binom{n}{i} q^{n-i} \\ \nonumber
    &= \frac{1}{q} \left( (q-1)^n - (-1)^{n-1} \binom{n}{n-1} q - (-1)^n \binom{n}{n} q^0 \right) \\ 
    &= \frac{(q-1)^n - (-1)^n}{q} + (-1)^n \cdot n.
\end{align}

Now we can estimate the number \( N^* \).

\begin{proposition} \label{estimation}
 Suppose $k,m\geq 2$, $n \geq 3$, $a,b, a_i\in \fq^*$, $\operatorname{char}(\fq) \nmid mk$, $n<mk$, and $\gcd(m,q-1)=1$. Then the number $N^*$ of nonzero $\mathbb{F}_q$--rational solutions of \eqref{MH} satisfies the following estimate:
  
 $$
  \left|N^*- \frac{(q-1)^n-(-1)^n}{q}\right| \le \begin{cases}
10(mk)^{4}q^{3/2} &\text{ if }\,\, n=3 \\
& \\
10(mk)^{4}q^{5/2} &\text{ if }\,\, n=4 \\
& \\
11(mk)^{n-1}q^{n-3/2}&\text{ if }\,\, n\geq 5.
 \end{cases}$$
\end{proposition}
\begin{proof}
From equations \eqref{O1} and \eqref{P1} we obtain
\begin{align*}
    N^* &= (N-q^{n-1})+ \sum_{i=1}^{n-2} (-1)^i\binom{n}{i} (N_i-q^{n-i-1}) \\
    & \quad + \frac{(q-1)^n-(-1)^n}q+(-1)^n \cdot n-(-1)^n \cdot n.
\end{align*}
Assume that $n\geq 5$. Using Theorem \ref{estimate} and Proposition \ref{estimatediagonal}, we obtain 
\begin{align}\label{estimateN*}\nonumber
    \left|N^{*}- \frac{(q-1)^n-(-1)^n}{q}\right|&\le \left|N-q^{n-1}\right|+ \sum_{i=1}^{n-2} \binom{n}{i}|N_i-q^{n-i-1}|\\
    &\le  9 (mk)^{n-1}q^{n-3/2}+ \sum_{i=1}^{n-2} \binom{n}{i}2 m^{n-i}q^{\frac{n-i-1}2}.
\end{align}

To derive a sharper estimate for $N^*$, we will derive an upper bound for the summation in the previous equation. For this purpose, we utilize the following fact
\begin{align*}\nonumber
    \left(1+\frac{1}{\sqrt q}\right)^n& =\sum_{i=0}^n \binom{n}{i}\left(\frac{1}{\sqrt{q}}\right)^i \\
\end{align*} that implies 
\begin{align}\nonumber 
    \sum_{i=1}^{n-2}\binom{n}{i}\left(\frac{1}{\sqrt q}\right)^i &= \left(1+\frac{1}{\sqrt q}\right)^n -1-n\left(\frac{1}{\sqrt q}\right)^{n-1}-\left(\frac{1}{\sqrt q}\right)^n\\\nonumber 
    &\le \left(1+\frac{1}{\sqrt q}\right)^n- \left(\frac{1}{\sqrt q}\right)^n \\\nonumber
    & =\frac{(\sqrt q+1)^n-1^n}{(\sqrt{q})^n}\le 2^n.
    \end{align}
Using this, we can have the following estimate: 
\begin{align}\label{combinatorio}
 \nonumber   \sum_{i=1}^{n-2}\binom{n}{i}2m^{n-i}q^{\frac{n-i-1}{2}} 
    & \le 2m^{n-1}q^{\frac{n-1}{2}}\left(\left(1+\frac{1}{\sqrt{q}}\right)^n-\left(\frac{1}{\sqrt{q}}\right)^n\right) \\ 
    & \le 2m^{n-1}q^{\frac{n-1}{2}} \cdot 2^n \le 4(2m)^{n-1}q^{\frac{n-1}{2}}.
    \end{align}
From Equations \eqref{estimateN*} and \eqref{combinatorio} and taking into account that $n \geq 5$, it follows that 
   \begin{align*}
       \left|N^* - \frac{(q-1)^n-(-1)^n}{q}\right| & \le 9(mk)^{n-1}q^{n-3/2}+4(2m)^{n-1}q^{\frac{n-1}{2}}\\
       & \le (mk)^{n-1}q^{\frac{n-1}2}\left(9q^{\frac{n-2}2}+4\right) \\
       & \le 11(mk)^{n-1}q^{\frac{2n-3}2} \\ &= 11(mk)^{n-1}q^{n-3/2}.
   \end{align*}

The bounds for $n=3$ and $n=4$ are obtained similarly.
\end{proof}

The estimate on $N^*$ gives the following sufficient condition for the existence of $\fq$--rational solutions with nonzero coordinates.

\begin{proposition} \label{existencia}
Suppose $k,m\geq 2$, $n \geq 5$, $a,b, a_i\in \fq^*$, $\operatorname{char}(\fq) \nmid mk$, $n<mk$, and $\gcd(m,q-1)=1$. If $q > 22^2(2mk)^{2n-2}$, then equation \eqref{MH} has at least one solution in $(\fq^{*})^n$.
\end{proposition}

\begin{proof}
From Proposition \ref{estimation} we deduce that
\begin{align*}
    N^* &\geq \dfrac{(q-1)^n-(-1)^n}{q}-11(mk)^{n-1}q^{n-3/2} \\
    & \geq\dfrac{(q-1)^n-1}{q}-11(mk)^{n-1}q^{n-3/2} \\
    & \geq \dfrac{(q-2)^n}{q} -11(mk)^{n-1}q^{n-3/2} \\
    & \geq q^{n-3/2} \Big(\dfrac{q^{1/2}}{2^n}-11(mk)^{n-1}\Big).
\end{align*}
Thus, $N^*>0$ if
$$\dfrac{q^{1/2}}{2^n}-11(mk)^{n-1}>0,$$
and therefore we conclude that \eqref{MH} has at least one solution in $(\fq^*)^n$ if
$q>22^2 (2mk)^{2n-2}$. 
\end{proof}

\begin{remark}
With similar arguments as above, if $n=3,4$ and $q>10^2 (mk)^8$, then $N^*>0.$
\end{remark}
\begin{remark}
   From Proposition \ref{existencia}  we deduce that $N^*$ exhibits the same asymptotic behavior in terms of $q$ as in the case $k=1$ presented in \cite[Proposition 4.7]{Melina&Mariana2020}.

Whereas Mordell's estimate applies only to the prime fields $\fp$ and the special case $a=0$
our estimate holds for every finite field $\fq$ and any values of $a$. Although our error term is of higher order than Mordell's, we supply one additional term in the asymptotic expansion of $N^*$ in terms of $q$ giving a sharper overall approximation in the more general setting.
\end{remark}

\section*{Acknowledgements}
This work was initiated at the workshop “Diversity in Finite Fields and Coding Theory,” held at IMPA (Rio de Janeiro, Brazil) from 7 – 12 July 2024.
The authors gratefully acknowledge the organizers—Carolina Araujo, Grasiele Jorge, Luciane Quoos, Cecília Salgado and Suely Lima, for the invitation and the excellent working conditions they provided.

\end{document}